\documentclass[12pt]{amsart}
\usepackage{amsmath,latexsym,amscd,amsbsy,amssymb,amsfonts,amsthm,fleqn,leqno,
euscript, graphicx, texdraw, pb-diagram}
\usepackage[all]{xy}
\usepackage{color}
\usepackage{tikz}
\numberwithin{equation}{section}

\newtheorem{thm}{Theorem}[section]
\newtheorem{pro}[thm]{Proposition}
\newtheorem{lem}[thm]{Lemma}

\theoremstyle{definition}

\theoremstyle{remark}
\newtheorem{claim}[thm]{Claim}

\begin{document}

\title[On $Q$-manifolds bundles]
{On $Q$-manifolds bundles}

\author{V. Valov}
\address{Department of Computer Science and Mathematics,
Nipissing University, 100 College Drive, P.O. Box 5002, North Bay,
ON, P1B 8L7, Canada} \email{veskov@nipissingu.ca}

\author{J. West}
\address{Department of Mathematics, Cornell University,
310 Malott Hall, Ithaca, New York 14853}
\email{west@math.cornell.edu}

\thanks{The first author was partially supported by NSERC
Grant 261914-19.}




\begin{abstract}
We prove a homological characterization of $Q$-manifolds bundles over $C$-spaces. This provides a partial answer to Question QM22 from \cite{w}.
\end{abstract}

\maketitle
\markboth{}{$Q$-bundles}



\section{Introduction}
The second author raised the question \cite[QM22]{w} whether there exists a \v{C}ech-cohomology version of the fibred general position theory of Toru\'{n}czyk-West \cite{tw} along the lines of Daverman-Walsh results \cite{dw} that detects the $Q$-manifold bundles among the $ANR$-fibrations with compact fibers.
In the present paper we present a homological version of Toru\'{n}czyk-West's \cite{tw} characterization of $Q$-manifold bundles over $C$-spaces.

 Following \cite{tw}, a {\em trivial $ANR$-fibration (resp., $AR$-fibration)} is a map $p:E\to B$ of paracompact Hausdorff spaces such that for some separable metric $ANR$-space (resp., $AR$-space) $X$ there is a closed embedding $i:E\hookrightarrow X\times B$ and a retraction $r:X\times B\to i(E)$ making the diagram commute
$$
\xymatrix{
E \ar[r]^{i}\ar[dr]_{p} & X\times B \ar[r]^{r}\ar[d]^{\pi_B} &  i(E)\ar[dl]^{\pi_B}\\
& B \\
}
$$
and a fiber preserving homotopy $H:X\times B\times\mathbb I\to X\times B$ from $r$ to $r_b\times\rm{id}_B$ for some $b\in B$, where $r_b:X\to i(p^{-1}(b))$ is defined by $r_b(x,b)=r((x,b))$.
A map $p:E\to B$ of paracompact Hausdorff spaces is {\em an $ANR$-fibration} if there is an open cover $\mathcal U$ of $B$ such that $p:p^{-1}(U)\to U$ is a trivial $ANR$-fibration for each $U\in\mathcal U$. We call the cover $\mathcal U$ a {\em local trivialization of $p$}. Note that every $U\in\mathcal U$ is also paracompact. If the maps $r$ and $H$ may be taken proper, $p$ is called a {\em proper $ANR$-fibration}.
If each $X_U$ is locally compact, compact, or complete, we say that $p:E\to B$ is a {\em locally compact, compact, or complete $ANR$-fibration}. Every complete trivial $ANR$-fibration has the following property: For every paracompact space $Z$, a closed subset $A\subset Z$ and maps $f:A\to E$ and $g:Z\to B$ such that $p\circ f=g|A$ there is a map $\widetilde g:V\to B$, where $V$ is a neighborhood of $A$ in $Z$, such that $p\circ\widetilde g=g|V$ and $\widetilde g$ extends $f$. A map $p:E\to B$ with this property is said to be a {\em locally soft map}. A map $p$ is {\em soft} if in the above definition $V$ is the whole space $Z$ \cite{sc}. Let us note that every locally soft map is open.

Let $p:E\to B$ be an $ANR$-fibration and $A\subset B$. We say that $p^{-1}(A)$ has the
{\em fibred disjoint $n$-disks property}
if each pair of fiber preserving maps $f,g:\mathbb I^n\times A\to p^{-1}(A)$ can be approximated arbitrary closely by fiber preserving maps
$f',g':\mathbb I^n\times A\to p^{-1}(A)$ with disjoint images, see \cite{tw}. Here, a map $f:\mathbb I^n\times A\to p^{-1}(A)$ is said to be {\em fiber preserving} provided $p(f(x,y))=y$ for all $(x,y)\in \mathbb I^n\times A$.
The fibred disjoint $2$-disks property is simply called {\em fibred disjoint disks property}. If there is an open cover $\mathcal U$ of $B$ such that $p^{-1}(U)$ has the fibred disjoint $n$-disks property for all $n$ and $U\in\mathcal U$, then $p$ satisfies the {\em fibred general position \rm{(FGP)}-property} \cite{tw}.

    Recall that a closed set $F\subset X$ is said to be a {\em $Z_n$-set} in $X$ if the set $C(\mathbb I^n,X\setminus F)$ is dense in $C(\mathbb I^n,X)$. Note that if $X$ is a metric $LC^{n-1}$-space, then a closed set $F\subset X$ is $Z_n$-set iff for each at most $n$-dimensional metric compactum $Y$ the set $\{f\in C(Y,X):f(Y)\cap F=\varnothing\}$ is dense in $C(Y,X)$, see \cite{b}.
We also say that a map $f:K\to X$ is a {\em $Z_n$-map} provided $f(K)$ is a $Z_n$-set in $X$.
 For consistency with \cite{to1},
instead of using the term $Z_{\infty}$ to mean $Z_n$ for all n as is often done,  we call $Z_\infty$-sets and $Z_\infty$-maps, respectively, $Z$-sets and $Z$-maps.

There are homological analogues of $Z_n$-sets and $Z_n$-maps. A closed set $F\subset X$ is called a {\em homological $Z_n$-set in $X$} if the singular homology groups $H_k(U,U\setminus F)$ are trivial for all open sets $U\subset X$ and all $k\leq n$, see \cite{bck}. It can be shown that if $X$ does not have isolated points then every homological $Z_n$-set in $X$ is nowhere dense. The homological $Z_n$-property is finitely additive and hereditary with respect to closed subsets \cite{bck}.
A map $f:K\to X$  is a {\em homological $Z_n$-map} provided the image $f(K)$ is a homological $Z_n$-set in $X$. Homological $Z_\infty$-sets were considered in \cite{dw} under the name sets of infinite codimension.

Combining \cite[Corollary 3.3]{to2} and \cite[Theorem 2.1]{bck}, we have the following:
\begin{pro}\label{z}
Let $X$ be an $LC^{n}$-space with $n\geq 2$. Then, a closed subset $A$ of $X$ is a $Z_n$-set in $X$ provided $A$ is an $LCC^1$ homological $Z_n$-set in $X$. Equivalently, $A$ is a $Z_n$-set iff it is a $Z_2$-set and a homological $Z_n$-set.
\end{pro}
All function spaces in this paper, if not explicitly stated otherwise, are equipped with the limitation topology, see \cite{bo1} and \cite{to3}.
Recall that a set $U\subset C(X,E)$ is open with respect to the limitation topology if for every $f\in U$ there is an open cover $\mathcal V$ of $Y$ such that $U$ contains the set
$B(f,\mathcal V)=\{g\in C(X,E):g{~}\mbox{is}{~}\mathcal V-\mbox{close to} f\}.$
Let $p:E\to B$ be an $ANR$-fibration, $K\subset B$ and $Z$ an arbitrary space. We denote by $C_{P}(Z\times K,p^{-1}(K))$ the set of all continuous fiber preserving maps from $Z\times K$ into $p^{-1}(K)$. One can show that each $f\in C_{P}(Z\times K,p^{-1}(K))$ is a perfect map provided $Z$ is compact.

\section{Preliminary results}
\begin{lem}\label{open}
Let $p:E\to B$ be a locally compact trivial $ANR$-fibration and $A\subset B$ be a closed set. Then for every space $K$ such that $K\times B$ is paracompact the restriction map
$q_A:C_P(K\times B,E)\to C_P(K\times A,p^{-1}(A))$, $q_A(f)=f|(K\times A)$, is open.
\end{lem}
\begin{proof}
Let $W\subset C_P(K\times B,E)$ be an open set and $f\in W$. Then there exists an open cover $\mathcal U$ of $E$ such that
$B(f,\mathcal U)\subset W$.
By \cite[Proposition (A10)]{tw}, there is an open cover $\mathcal V$ of $E$ such that any map $h\in C_P(K\times A,p^{-1}(A))$,
which is $\mathcal V$-close to $f|(K\times A)$, is extended to a map $\widetilde h\in B(f,\mathcal U)$. This means that the set
$B(q_A(f),\mathcal V)$ is contained in $q_A(W)$. Since the interior
of $B(q_A(f),\mathcal V)$ in $C_P(K\times A,p^{-1}(A))$ contains $q_A(f)$ (see \cite{bo1}), $q_A(W)$ is  open in $C_P(K\times A,p^{-1}(A))$.
\end{proof}

 In case $p:E\to B$ is a trivial locally compact fibration,
we need another description of the limitation topology on $C_P(K\times B,E)$.
If $(X,d)$ is the associated with $p$ separable metric $ANR$-space, we consider the
pseudo-metric $d_E$ on $E$ defined by
$d_E(x,y)=d(\pi_X(x),\pi_X(y))$, where $\pi_X:X\times B\to X$ is the projection.
\begin{lem}\label{limitation}
Let $p:E\to B$ be a trivial $ANR$-fibration and $(X,d)$ be a locally compact separable metric space associated with $p$.
Then for any space $K$ a set $U\subset C_P(K\times B,E)$ is open with respect to the limitation topology iff for every $f\in U$ there is a continuous function $\varepsilon:E\to (0,1]$ such that $U$ contains the set
$$\displaystyle B_{d_E}(f,\varepsilon)=\{g\in C_P(K\times B,E):d_E(f(y),g(y))<\varepsilon (f(y)) \forall y\in K\times B\}.$$
\end{lem}
\begin{proof}
Suppose $U\subset C_P(K\times B,E)$ is open with respect to the limitation topology and $f\in U$. Then there is an open cover $\mathcal V$ of $E$ such that $U$ contains the set
$B(f,\mathcal V)$. Since $X\times B$ is paracompact, so is $E$.
Let $r:X\times B\to E$ be a fiber-preserving retraction and $\mathcal W$ be a locally finite open cover of $X\times B$ which is star-refinement of the cover $\widetilde{\mathcal V}=r^{-1}(\mathcal V)$. We fix a countable open cover $\{X_i\}_{i\geq 1}$ of $X$ such that each $X_i$ has a compact closure and $\overline X_i\subset X_{i+1}$. 
Since $\overline X_i$ is compact, every $b\in B$ has a neighborhood $O_b^i$ in $B$ such that for every $x\in\overline X_i$ there is $W_x\in\mathcal W$ and a neighborhood $O_x$ of $x$ in $X$ with
$O_x\times O_b^i\subset W_x$.
Choose a bounded continuous pseudo-metric $\sigma_i\leq 1$ on $B$ such that the family of all open balls $\{B_{\sigma_i}(b,1):b\in B\}$ refines the cover
$\{O_b^i:b\in B\}$. Let
$\rho$ be the pseudo-metric on $X\times B$ defined by $\rho((x,b),(x',b'))=d(x,x')+\sigma(b,b')$, where
$\sigma(b,b')=\sum_{i=1}^\infty\sigma_i(b,b')/2^i$.
Define a function $\varepsilon$ on $E$ by
$$\varepsilon(z)=\rho(z,E\backslash\rm{St}(z,\mathcal W)).$$  We can assume that $\varepsilon\leq 1$. Clearly $\varepsilon$ is continuous and bounded, and it is positive. Indeed,
suppose $\varepsilon(z)=0$ for some $z=(x,b)\in E$, where $x\in X_j$. Then there is a sequence $\{z_n\}\subset E\backslash\rm{St}(z,\mathcal W)$, $z_n=(x_n,b_n)$, such that $\rho(z,z_n)<1/n$ for all $n$. Then $d(x,x_n)<1/n$ and, since $d$ is a metric on $X$, the sequence $\{x_n\}$ converges to $x$. We can assume that each $x_n\in X_j$. On the other hand, $\sigma(b,b_n)<1/n$ implies that $\sigma_j(b,b_n)<2^j/n$ for all $n$. Hence, there exists $n_0$ with $2^j/n<1$ for every $n\geq n_0$. Consequently, $b,b_n\in O_{b^*}^j$ for any $n\geq n_0$ and some $b^*\in B$. According to the construction of the neighborhoods $O_b^j$, we have $O_x\times O_{b^*}^j\subset W_x$ for some $W_x\in\mathcal W$ and a neighborhood $O_x$ of $x$. Now, take $m>n_0$ with $x_m\in O_x$, and observe that $z, z_m\in O_x\times O_{b^*}^j\subset W_x$. Hence, $z_m\in\rm{St}(z,\mathcal W)$, a contradiction.
Moreover, for any $z\in E$ the inequality
$\rho(z,z')<\varepsilon(z)$ implies that $z,z'\in W$ for some $W\in\mathcal W$. Hence, $B_{\rho}(f,\varepsilon)\subset B(f,\rm{St}\mathcal W)\subset B(f,\mathcal V)$. Because $\displaystyle B_{d_E}(f,\varepsilon)=B_{\rho}(f,\varepsilon)$, we finally have
$\displaystyle B_{d_E}(f,\varepsilon)\subset U$.

To show the other implication of Lemma \ref{limitation}, let $\displaystyle B_{d_E}(f,\varepsilon)\subset U$ for some positive continuous function $\varepsilon$. For every $z\in E$ let
$$W_z=\{e\in E:d_E(e,z)<\varepsilon(z)/4{~}\mbox{and}{~}\varepsilon(z)/2<\varepsilon(e)<3\varepsilon(z)/2\}.$$ Then, for the open cover
$\mathcal V=\{W_z:z\in E\}$ of $E$
we have $B(f,\mathcal V)\subset\displaystyle B_{d_E}(f,\varepsilon)$.
\end{proof}

We need to consider another topology on function spaces $C(X,Y)$, the so called {\em source limitation topology} \cite{bv}. The local base in that topology at given $h:X\to Y$  consists of all sets
$$\Lambda(h,\eta)=\{h'\in C(X,Y):\rho(h(x),h'(x))<\eta(x){~} \forall x\in X\},$$ where $\rho$ is a continuous pseudo-metric on $Y$ and $\eta$ is a continuous and positive function on $X$. This topology, called sometimes the fine topology, was initially introduced for metrizable spaces $Y$ (in this case $\rho$ is a fixed compatible metric on $Y$), see \cite{mc}, \cite{mu}. Clearly, the source limitation topology is stronger than the limitation one. This topology has the Baire property provided $X$ is paracompact and $Y$ is completely metrizable, see \cite{mu}.

\begin{pro}\label{baire}
Let $p:E\to B$ be a  locally compact $ANR$-fibration.
Then for any compact space $K$ the space $C_P(K\times B,E)$ has the Baire property.
\end{pro}

\begin{proof}
According to \cite[Proposition A9.4]{tw}, there is a closed fiber-preserving embedding of $E$ into $Q\times [0,\infty)\times B$.
Let $C_P(K\times B,Q\times [0,\infty)\times B)$ be the set of all fiber-preserving maps into $Q\times [0,\infty)\times B$. Then $C_P(K\times B,E)$
is a closed subset of $C_P(K\times B,Q\times [0,\infty)\times B)$ with respect to the limitation topology. So, it suffices to show that $C_P(K\times B,Q\times [0,\infty)\times B)$ equipped with the limitation topology has the Baire property. Therefore, we can suppose that $E=X\times B$ and $p:E\to B$ is the projection, where $X$ is a locally compact separable metric space with a fixed complete metric $d$.
Observe that the correspondence $\theta:f\rightsquigarrow f_X=\pi_X\circ f$ between $C_P(K\times B,E)$ and $C(K\times B,X)$ is bijective.
Since $C(K\times B,X)$ with the source limitation topology has the Baire property, the proof is reduced to showing that $\theta$ is a homeomorphism when $C_P(K\times B,E)$ carries the limitation topology and $C(K\times B,X)$ is equipped with the source limitation topology.
Let show that both $\theta^{-1}$ and $\theta$ are continuous.

Suppose $\{f_X^\alpha\}$ is a net in $C(K\times B,X)$ converging with respect to the source limitation to $f_X=\pi_X\circ f$ for some
$f\in C_P(K\times B,E)$. To show that $\{f^\alpha\}$ converges to $f$, let $U$ be a neighborhood of $f$ in $C_P(K\times B,E)$. Then, by
Lemma \ref{limitation}, there exists a continuous positive function $\varepsilon$ on $E$ such that
$\displaystyle B_{d_E}(f,\varepsilon)\subset U$.
Then there is $\alpha_0$ such that $d(f_X(z),f_X^\alpha(z))<\varepsilon(f(z))$ for each $z\in K\times B$ and $\alpha\geq\alpha_0$. This implies $f^\alpha\in B_{d_E}(f,\varepsilon)\subset U$ for all $\alpha\geq\alpha_0$. Hence, $\theta^{-1}$ is continuous.

Suppose now that $\{f^\alpha\}$ converges to $f$ in $C_P(K\times B,E)$ with respect to the limitation topology, and let $\varepsilon'$ be a continuous positive function on $K\times B$. Since $f$ is a perfect map, we find a continuous positive function $\varepsilon''$ on $f(K\times B)$ such that $\varepsilon''(f(z))<\varepsilon'(z)$, $z\in K\times B$, (see the proof of Claim 4.2 below) and then extend $\varepsilon''$ to a continuous positive function $\varepsilon$ on $E$ (this can be done because $E$ is paracompact and $f(K\times B)$ is closed in $E$).
The set $\displaystyle B_{d_E}(f,\varepsilon)$ is a neighborhood of $f$ with respect to the limitation topology,  see the proof of Lemma \ref{limitation}.
 Therefore, there is $\alpha_0$ with $f^\alpha\in\displaystyle B_{d_E}(f,\varepsilon)$ for all $\alpha\geq\alpha_0$. This means that $d(f_X(z),f_X^\alpha(z))<\varepsilon(f(z))<\varepsilon'(z)$ for all
$\alpha\geq\alpha_0$ and $z\in K\times B$. Hence, $\theta$ is continuous.
\end{proof}

For a space $Y$ let $\mathcal F(Y)$ denote the family of all non-empty closed subsets of $Y$. If $Y$ is a convex subspace of a linear space, then $\mathcal F_c(Y)$ stands for the closed convex subsets of $Y$. A set-valued map $\varphi:X\to\mathcal F(Y)$ is said to be lower semi-continuous, or l.s.c., if
$\varphi^{-1}(U)=\{x\in X:\varphi(x)\cap U\neq\varnothing\}$ is open in $X$ for every open $U\subset Y$.

\begin{pro}\label{approximation}
Let $p:E\to B$ be a locally compact $ANR$-fibration. Then the fibration $p':E\times Q\to B$ has the \rm{FGP}-property.
\end{pro}

\begin{proof}
Without loss of generality, we may assume that $p$ is a trivial fibration.
We need to show that every fiber preserving map $f:Q\times\{1,2\}\times B\to E\times Q$ can be approximated in the limitation topology by
maps $g\in C_P(Q\times\{1,2\}\times B,E\times Q)$ such that $g(Q\times\{1\}\times B)\cap g(Q\times\{2\}\times B)=\varnothing$.
Here $C_P(Q\times\{1,2\}\times B,E\times Q)$ is the set of all $p'$-preserving maps from $C(Q\times\{1,2\}\times B,E\times Q)$.
To this end, we fix $f\in C_P(Q\times\{1,2\}\times B,E\times Q)$ and an open cover $\mathcal U$ of $E\times Q$. Let $f_E, f_Q$ be the maps $\pi_E\circ f$ and $\pi_Q\circ f$, where $\pi_E$ and $\pi_Q$ are the projections of $E\times Q$ onto $E$ and $Q$, respectively.
Since $E$ is paracompact, so is
$E\times Q$. 
Using that $Q$ is compact, every $e\in E$ has a neighborhood $O_e$ such that for every $x\in Q$ there is $V\in\mathcal U$ with $O_e\times\{x\}\subset V$. As in the proof of Lemma \ref{limitation}, we can find a bounded continuous pseudo-metric $\sigma$ on $E$ and a continuous bounded positive function
$\varepsilon$ on $E\times Q$ such that any $g\in C_P(Q\times\{1,2\}\times B,E\times Q)$ is $\mathcal U$-close to $f$ provided
$\rho(f(y),g(y)<\varepsilon(f(y))$ for all $y\in Q\times\{1,2\}\times B$, where $\rho$ be the pseudo-metric on $E\times Q$ defined by $\rho((e,x),(e',x'))=d(x,x')+\sigma(e,e')$ (here $d$ is the ordinary metric on $Q$).

We are going to show there is a function $g_Q:Q\times\{1,2\}\times B\to Q$ such that $g_Q(Q\times\{1\}\times B)\cap g_Q(Q\times\{2\}\times B)=\varnothing$ and
$d(f_Q(y),g_Q(y))<\varepsilon(f(y))$ for all $y\in Q\times\{1,2\}\times B$.
Since $\varepsilon$ is bounded, there is a continuous extension $\widetilde\varepsilon:\beta(E\times Q)\to [0,\infty)$, where $\beta(E\times Q)$
is the \v{C}ech-Stone extension of $E\times Q$. Let also $\widetilde f:\beta(Q\times B)\times\{1\}\oplus\beta(Q\times B)\times\{2\}\to\beta(E\times Q)$ and
$\widetilde\pi_Q:\beta(E\times Q)\to Q$ be the extensions of $f$ and $\pi_Q$.  (We use $\oplus$ to denote the disjoint union.) Then $\widetilde E=\widetilde\varepsilon^{-1}(0,\infty)$ and $Y=\widetilde f^{-1}(\widetilde E)$ are locally compact and $\sigma$-compact spaces. We represent $Y$ as the union of the sets
    $K_i=Y_i\times\{1\}\oplus Y_i\times\{2\}$, where $Y_i$ is an increasing sequence of compact subspaces of $\beta(Q\times B)$. Let $\delta=\widetilde\varepsilon\circ\widetilde f:Y\to (0,\infty)$ and $\widetilde f_Q=\widetilde\pi_Q\circ\widetilde f:Y\to Q$. We consider the space $C(Y,Q)$ with the source limitation topology.
Recall that the local base in that topology at given $h:Y\to Q$  consists of all sets
$$\Lambda(h,\eta)=\{h'\in C(Y,Q):d(h(y),h'(y))<\eta(y){~} \forall y\in Y\},$$ where $\eta$ is a continuous and positive function on $Y$.
As we already noted, this topology has the Baire property.
\begin{claim}
All restriction maps $\theta_i:C(Y,Q)\to C(K_i,Q)$, $\theta_i(h)=h|K_i$, are surjective and open in the source limitation topology.
\end{claim}
Indeed, if $W\subset C(Y,Q)$ is open let $h'\in\theta_i(W)$. Then $h'=\theta_i(h)$ for some $h\in W$ and there is
$\eta\in C(Y,(0,\infty)$ with $\Lambda(h,\eta)\subset W$. Consider the open in $C(K_i,Q)$ set
$$\Lambda_i(h',\eta/2)=\{g\in C(K_i,Q):d(h'y),g(y))<\eta(y)/2{~} \forall y\in K_i\}.$$
If $g\in\Lambda_i(h',\eta/2)$, we define the lower semi-continuous set valued map $\Phi:Y\to\mathcal F_c(Q)$,
$\Phi(y)=\{z\in Q:d(z,h(y)\leq\eta(y)/2\}$ if $y\not\in K_i$ and $\Phi(y)=g(y)$ if $y\in K_i$. Then by Michael's selection theorem \cite{m},
there is a continuous function $\widetilde g\in C(Y,Q{\color{blue})}$ with $\widetilde g(y)\in\Phi(y)$ for all $y\in Y$. Obviously,
$\widetilde g\in\Lambda(h,\eta)\subset W$ and $\theta_i(\widetilde g)=g$. This means that $\Lambda_i(h',\eta/2)\subset\theta_i(W)$. Hence,
each $\theta_i$ is open. Surjectivity of $\theta_i$ is obvious because $Q$ is an absolute retract for the normal spaces.

Now, for every $i$ let $G_i$ be the set of all maps $h\in C(K_i,Q)$ such that $h(Y_i\times\{1\})\cap h(Y_i\times\{2\})=\varnothing$. Using that $K_i$ are compact, one can show that  each $G_i$ is open and dense in
$C(K_i,Q)$ (note that the source limitation topology on $C(K_i,Q)$ coincides with the compact open topology). Hence,
$G=\bigcap_{i=1}^\infty\theta_i^{-1}(G_i)$ is dense in $C(Y,Q)$ with respect to the source limitation topology. Since $\Lambda(\widetilde f_Q,\delta)$ is open in that topology, there is $\widetilde g_Q\in G\cap\Lambda(\widetilde f_Q,\delta)$. Let $g_Q=\widetilde g_Q|(Q\times\{1,2\}\times B)$ and define
$g:Q\times\{1,2\}\times B\to E\times Q$ by $g(y)=(f_E(y),g_Q(y))$. Clearly, $g$ is $p'$-preserving and
$\rho(f(y),g(y))=d(f_Q(y),g_Q(y))<\varepsilon(f(y))$ for all $y\in Q\times\{1,2\}\times B$. This means that $g$ is $\mathcal U$-close to $f$. Moreover, $\widetilde g_Q\in G$ implies $g(Q\times\{1\}\times B)\cap g(Q\times\{2\}\times B)=\varnothing$.
\end{proof}

\section{Homological characterization of $Q$-manifold bundles over $C$-spaces}
In this section we prove a homological characterization of $Q$-manifold bundles over $C$-spaces. This provides a partial answer to Question QM22 from \cite{w}.

The $C$-space property was originally defined by Haver \cite{ha} for compact metric spaces. Addis and Gresham \cite{ag} reformulated Haver's definition for arbitrary spaces: A space $X$ has property $C$ if for every sequence $\{\mathcal U_n\}$ of open covers of $X$ there exists a sequence
    $\{\mathcal V_n\}$ of open disjoint families in $X$ such that each $\mathcal V_n$ refines $\mathcal U_n$ and $\bigcup_{n\geq 1}\mathcal V_n$ is a cover of $X$. Every finite-dimensional paracompact space, as well as every countable-dimensional (a countable union of finite-dimensional sets) metric space, is a $C$-space \cite{ag}, but there is a compact metric $C$-space which is not countable-dimensional \cite{po}. On the other hand, normal $C$-spaces are weakly infinite-dimensional in the sense of Alexandroff \cite{ag}. We say that $X$ is an {\em hereditary $C$-space} if every open subset of $X$ is a $C$-space. For example, every paracompact $C$-space whose open sets are $F_\sigma$ is an hereditary $C$-space.

Although not formulated in this form, Lemma \ref{gv} below was actually established in \cite[Proposition 3.1]{gv}.
\begin{lem}\cite{gv}\label{gv}
Let $Y$ be a closed convex subset of a Banach space and $V\subset Y$ be open. Suppose $X$ is a paracompact $C$-space and $\varphi:X\rightsquigarrow\mathcal F_c(Y)$ is l.s.c. such that $\varphi(x)\subset V$ for all $x\in X$. Then for every set-valued map
$\psi:X\to\mathcal F(V)$ with a closed graph (in $X\times V$) there is a map $f:X\to V$ with $f(x)\in\varphi(x)\backslash\psi(x)$ for all $x\in X$ provided each $\varphi(x)\cap\psi(x)$ is a $Z$-set in $\varphi(x)$.
\end{lem}

\begin{pro}\label{fibred_position}
Let $p:E\to B$ be a locally compact $ANR$-fibration with compact $Q$-manifold fibers such that $B$ is a locally compact $C$-space. Then
every $b\in B$ has a basis of neighborhoods $U$ with compact closures such that the fibrations $p|p^{-1}(\overline U):p^{-1}(\overline U)\to \overline U$ have the \rm{FGP}-property.
\end{pro}
\begin{proof}
By Proposition \ref{approximation}, the fibration $E\times Q\to B$ has the \rm{FGP}-property.
Hence, according to \cite[Theorem 2.3]{tw}, $p':E\times Q\to E\to B$ is a $Q$-manifold bundle. Consequently,
$p'':E\times Q\times [0,1)\to E\to B$ is also a $Q$-manifold bundle. Let $\mathcal U$ be a simultaneous local trivialization of $p$ and
$p'$, and hence of $p''$.
It is clear that the proof is reduced to show that for every compact set $K\subset B$ which is contained in some $U\in\mathcal U$ the fibration $p|p^{-1}(K):p^{-1}(K)\to K$ has the \rm{FGP}-property. To this end, we can assume that $p:E\to B$ is a trivial $ANR$-fibration over a compact space $B$ such that both $p':E\times Q\to E\to B$ and $p'':E\times Q\times [0,1)\to E\to B$ are trivial $Q$-manifold bundles. Therefore, we need
to show that for every $n$ every fiber-preserving map $f:\mathbb I^n\times\{1,2\}\times B\to E$ can be approximated by fiber-preserving maps $f':\mathbb I^n\times\{1,2\}\times B\to E$ such that
 $f'(\mathbb I^n\times\{1\})\cap f'(\mathbb I^n\times\{2\})=\varnothing$.
To do this, take a fiber-preserving homeomorphism
$$h:E\times Q\times [0,1)\to B\times M,$$ where $M$ is a $Q$-manifold homeomorphic to $p^{-1}(b)\times Q\times [0,1)$ for all $b\in B$.
Since $M$ is a product of the $Q$-manifolds $p^{-1}(b)\times Q$ and $[0,1)$, it can be embedded as an open subset of a copy $Q_M$ of the Hilbert cube \cite [Corollary 7.4.4(2)]{vm}.  Since the domains of all function spaces in the present proof are compact spaces, the limitation topology coincides with the compact open topology. Therefore, all function spaces will be considered with the compact open topology. 
We also consider the projection $\pi_E:E\times Q\times [0,1)\to E$. Then $\pi_E$ generates the continuous map
$$\pi_E^*:C(\mathbb I^n\times\{1,2\},E\times Q\times [0,1))\to C(\mathbb I^n\times\{1,2\},E),$$ $\pi_E^*(w)=\pi_E\circ w$.
Since, $C(\mathbb I^n\times\{1,2\},E\times Q\times [0,1))$ is homeomorphic to the product $C(\mathbb I^n\times\{1,2\},E)\times C(\mathbb I^n\times\{1,2\},Q\times [0,1))$, $\pi_E^*$ is the projection onto $C(\mathbb I^n\times\{1,2\},E)$.
For every $b\in B$ let $C(b)=C(\mathbb I^n\times\{1,2\},p^{-1}(b)\times Q\times [0,1))$ and $C=\bigcup\{C(b):b\in B\}$ considered as a subspace of $C(\mathbb I^n\times\{1,2\},E\times Q\times [0,1))$. Let also
$L(b)$ be the subspace of $C(b)$ consisting of all maps $g$ such that
$\pi_E^*(g)(\mathbb I^n\times\{1\})\cap \pi_E^*(g)(\mathbb I^n\times\{2\})\neq\varnothing$. Actually, $L(b)$ is the product $L_1(b)\times C(\mathbb I^n\times\{1,2\},Q\times [0,1))$, where $L_1(b)$ is the set of all $g'\in C(\mathbb I^n\times\{1,2\},p^{-1}(b))$ with
$g'(\mathbb I^n\times\{1\})\cap g'(\mathbb I^n\times\{2\})\neq\varnothing$.

\begin{claim}\label{3}
$L(b)$ is a $Z$-set in $C(b)$ for each $b\in B$.
\end{claim}
Since $L(b)$ is homeomorphic to the product $L_1(b)\times C(\mathbb I^n\times\{1,2\},Q\times [0,1))$ and $C(b)$ is homeomorphic to
$C(\mathbb I^n\times\{1,2\},p^{-1}(b))\times C(\mathbb I^n\times\{1,2\},Q\times [0,1))$, it suffices to show that $L_1(b)$ is a $Z$-set in $C(\mathbb I^n\times\{1,2\},p^{-1}(b))$.
Obviously, $L_1(b)$ is closed in $C(\mathbb I^n\times\{1,2\},p^{-1}(b))$.
So, we need to show that every map $u:Q\to C(\mathbb I^n\times\{1,2\},p^{-1}(b))$ is approximated by maps $u':Q\to C(\mathbb I^n\times\{1,2\},p^{-1}(b))$ with $u'(Q)\cap L_1(b)=\varnothing$. Since $Q$ is compact, the exponential map
$$\Lambda_b:C(Q\times \mathbb I^n\times\{1,2\},p^{-1}(b))\to C(Q,C(\mathbb I^n\times\{1,2\},p^{-1}(b))),$$ defined by $\{[\Lambda_b(v)](q)\}(x)=v(q,x)$, $q\in Q$ and
$x\in \mathbb I^n\times\{1,2\}$, is a homeomorphism, see \cite[Theorem 3.4.3]{en}. So, $v=\Lambda_b^{-1}(u)\in C(Q\times \mathbb I^n\times\{1,2\},p^{-1}(b))$ for every $u\in C(Q,C(\mathbb I^n\times\{1,2\},p^{-1}(b)))$.
Because $p^{-1}(b)$ is a $Q$-manifold, $v$ can be approximated by maps
$v'\in C(Q\times \mathbb I^n\times\{1,2\},p^{-1}(b))$ such that $v'(Q\times \mathbb I^n\times\{1\})\cap v'(Q\times \mathbb I^n\times\{2\})=\varnothing$. Therefore, $u'(Q)\cap L_1(b)=\varnothing$, where $u'=\Lambda_b(v')$.
This completes the proof of the claim.

Next, we consider the set $L=\bigcup\{L(b):b\in B\}\subset C$. It is easily seen that $L$ is closed in $C$ and that $\widetilde p:C\to B$,
$\widetilde p(g)=b$ for all $g\in C(b)$, defines a continuous map.
The fiber-preserving homeomorphism $h$ provides a homeomorphism between the sets $p^{-1}(b)\times Q\times [0,1)$ and $\{b\}\times M$ for every $b\in B$. Consequently, $h$ generates a homeomorphism
$$h^*:C(\mathbb I^n\times\{1,2\},E\times Q\times [0,1))\to C(\mathbb I^n\times\{1,2\},B\times M)$$ such that $h^*(g)\in C(\mathbb I^n\times\{1,2\},\{b\}\times M)$ for all
$g\in C(b)$. On the other hand, there is a natural homeomorphism between $C(\mathbb I^n\times\{1,2\},B\times M)$ and the product
$C(\mathbb I^n\times\{1,2\},B)\times C(\mathbb I^n\times\{1,2\},M)$. Therefore, we have the
commutative diagram
$$
\xymatrix{
C \ar[r]^{h^*}\ar[dr]_{\widetilde p} & B\times V \ar[d]^{\pi_B}\\
& B \\
}
$$
with $h^*(C(b))=\{b\}\times V$, where $V=C(\mathbb I^n\times\{1,2\},M)$.
Let us show that any fiber-preserving map $f\in C_P(\mathbb I^n\times\{1,2\}\times B,E)$ can be approximated by fiber-preserving maps $g:\mathbb I^n\times\{1,2\}\times B\to E$ such that $g(\mathbb I^n\times\{1\}\times B)\cap g(\mathbb I^n\times\{2\}\times B)=\varnothing$. To this end, fix
a map $f\in C_P(\mathbb I^n\times\{1,2\}\times B,E)$ and consider
the exponential homeomorphism $$\Lambda: C(\mathbb I^n\times\{1,2\}\times B,E)\to C(B,C(\mathbb I^n\times\{1,2\},E)).$$
Then $\Lambda(f)(b)\in C(\mathbb I^n\times\{1,2\}, p^{-1}(b))$ for all $b\in B$. Therefore, $\Lambda$ is a homeomorphism between
$C_P(\mathbb I^n\times\{1,2\}\times B,E)$ and the set $\Omega$ consisting of all
$\{l\in C(B,C(\mathbb I^n\times\{1,2\},E))$ with $l(b)\in C(\mathbb I^n\times\{1,2\},p^{-1}(b))$ for every $b\in B$. We identify $E$ with $E\times\{\overline 0\}\times 0$ and each $p^{-1}(b)$ with $p^{-1}(b)\times\{\overline 0\}\times 0$, where $\overline 0=(0,0,...)\in Q$.
We also consider the retraction $r_E:E\times Q\times [0,1)\to E\times\{\overline 0\}\times\{0\}$ defined by $r_E(z_1,z_2,z_3)=(z_1,\overline 0,0)$.
This retraction generates a continuous map
$$r_E^*:C(\mathbb I^n\times\{1,2\},E\times Q\times [0,1))\to C(\mathbb I^n\times\{1,2\},E\times \{\overline 0\}\times\{0\}).$$
Therefore, we identify $C(\mathbb I^n\times\{1,2\},E)$ with  $C(\mathbb I^n\times\{1,2\},E\times \{\overline 0\}\times\{0\})$ and every
$C(\mathbb I^n\times\{1,2\},p^{-1}(b))$ with  $C(\mathbb I^n\times\{1,2\},p^{-1}(b)\times \{\overline 0\}\times\{0\})$.
 So, every $\Lambda(f)(b)$ can be considered as an element of $C(b)$.
 This means that $\Lambda(f)$ is a map from $B$ to $C$. Moreover, one can show that $\Lambda(f):B\to C$ is an embedding with
 $\Lambda(f)(b)\in\widetilde p^{-1}(b)=C(b)$ for all $b\in B$. So is the composition
 $u=h^*\circ\Lambda(f):B\to B\times V$ such that $u(b)\in\pi_B^{-1}(b)$, $b\in B$.
\begin{claim}\label{4}
$f$ can be approximated by maps $g\in C_P(\mathbb I^n\times\{1,2\}\times B,E)$ such that
$g(\mathbb I^n\times\{1\}\times B)\cap g(\mathbb I^n\times\{2\}\times B)=\varnothing$.
\end{claim}

Let $O_{\Lambda(f)}=\bigcap_{i=1}^{k}<P_i,W_i>$ be a neighborhood of $\Lambda(f)$ in the space $C(B,C(\mathbb I^n\times\{1,2\},E\times\{\overline 0\}\times\{0\}))$ with respect to the compact-open topology. Here, $P_i\subset B$ are compact sets, $W_i$ are open in $C(\mathbb I^n\times\{1,2\},E\times\{\overline 0\}\times\{0\})$ and $<P_i,W_i>$ consists of all maps $l\in C(B,C(\mathbb I^n\times\{1,2\},E\times\{\overline 0\}\times\{0\}))$ such that $l(P_i)\subset W_i$.
 The sets
$\widetilde W_i=(r_E^*)^{-1}(W_i)$ are open in $C(\mathbb I^n\times\{1,2\},E\times Q\times [0,1))${\color{blue}.  S}o are the sets $h^*(\widetilde W_i)$ in
$C(\mathbb I^n\times\{1,2\},B\times M)$. Hence, each $G_i=h^*(\widetilde W_i)\cap (B\times V)$ is open in $B\times V$. Since $\Lambda(f)(P_i)\subset W_i$, $u(P_i)\subset G_i$ for all $i$. Because $M$ is open in $Q_M$ and $Y=C(\mathbb I^n\times\{1,2\},Q_M)$ is homeomorphic to a closed convex subset of the Banach space $\big(C(\mathbb I^n\times\{1,2\},l_2),||.||\big)$, $V=C(\mathbb I^n\times\{1,2\},M)$ is homeomorphic to an open sunset of $Y$. Therefore, the continuous function
$\alpha:B\to\mathbb R$, $\alpha(b)=\rm{dist}(\pi_V(u(b)),Y\backslash V)$ is positive, where $\pi_V:B\times V\to V$ is the projection.
Using that $\{P_i:i\leq k\}$ is a finite family of compact sets in $B$ with $u(P_i)\subset G_i$ and each $\pi_B|u(P_i):u(P_i)\to P_i$ being a homeomorphism, one can show that there exists $m$ such that
the sets $\displaystyle T_b=\{(b,v)\in B\times V:||v-\pi_V(u(b))||<\frac{\alpha(b)}{m}\}$ are contained in $G_i$ for every $b\in P_i$ and $i\leq k$. Now, we consider the set-valued map
$$\displaystyle\varphi: B\rightsquigarrow\mathcal F_c(Y),{~}
\varphi(b)=\{v\in V:||v-\pi_V(u(b))||\leq\frac{\alpha(b)}{N}\},$$
where $N>m$. Obviously, $\{b\}\times\varphi(b)\subset G_i$ for all $b\in P_i$, $i\leq k$. Moreover, $\varphi$ is lower semi-continuous. Observe also that the set $h^*(L)$ is closed in $B\times V$ and $h^*(L(b))$ is closed in $h^*(C(b))=\{b\}\times V$ for all $b\in B$. So, $\pi_V(h^*(L(b)))$ is a closed subset of $V$ homeomorphic to $h^*(L(b))$, and by Claim \ref{3}, it is a $Z$-set in $V$. Hence, we have the set-valued map $\psi:B\rightsquigarrow\mathcal Z(V)$,
$\psi(b)=\pi_V(h^*(L(b)))$,
where $\mathcal Z(V)$ is the family of all $Z$-sets in $V$.

We claim that the graph $Gr(\psi)$ of $\psi$ is closed in $B\times V$. Indeed, suppose $\{(b_\gamma,v_\gamma)\}\in Gr(\psi)$ is a net converging to $(b_0,v_0)$ in $B\times V$. Since $h^*$ is a homeomorphism between $C$ and $B\times V$, there is a net $\{g_\gamma\}\subset C$ converging to some $g_0\in C$ such that $g_\gamma\in L(b_\gamma)$ and $h^*(g_\gamma)=(b_\gamma,v_\gamma)$. Since $L$ is closed in $C$, $g_0\in L$. On the other hand, $\{b_\gamma\}$ converges to $b_0$ implies that $g_0\in C(b_0)$. Hence $g_0\in L\cap C(b_0)=L(b_0)$. Finally, $h^*(g_0)=(b_0,v_0)\in Gr(\psi)$.

Since each set $O_b=\displaystyle\{v\in V:||v-\pi_V(u(b))||<\frac{\alpha(b)}{N}\}, {~}b\in B,$ is open in $V$, $\psi(b)
\cap O_b$ is a $Z$-set in $O_b$. This implies that $\psi(b)\cap\varphi(b)$ is a $Z$-set in $\varphi(b)$ for each $b$.
(This follows because radial contraction toward $\pi_V(u(b)$ shows that the boundary of $\phi(b)$ is a Z-set in $\phi(b)$, and as
$\phi(b)\cap O_b$ is a countable union of Z-sets of $C(b)$.  Thus $\psi(b)$ is a countable union of Z-sets of $\phi(b)$.)
Therefore, we may apply Lemma \ref{gv} to find a map $\phi:B\to V$ with $\phi(b)\in\varphi(b)\backslash\psi(b)$ for all $b\in B$.
Then the equality $u'(b)=(b,\phi(b))$ provides a map $u':B\to B\times V$ such that $u'(P_i)\subset G_i$ for all $i\leq k$ and
$u'(B)\cap h^*(L)=\varnothing$. Hence, each $(h^*)^{-1}(u'(b))$ belongs to $C(b)\backslash L$. Consequently, we have a map
$g':B\to C\backslash L$, $g'(b)=(h^*)^{-1}(u'(b))$,
 with $g'(P_i)\subset\widetilde W_i$, $i\leq k$. The composition $g''=r_E^*\circ g'$ provides a map from $B$ into $C_P(\mathbb I^n\times\{1,2\},E\times\{\overline 0\}\times 0))$ such that for every $b\in B$ we have:
\begin{itemize}
\item[(3)] $g''(b)\in C(\mathbb I^n\times\{1,2\},p^{-1}(b)\times\{\overline 0\}\times 0))$;
\item[(4)] $g''(P_i)\subset W_i$, $i\leq k$;
\item[(5)] $g''(b)(\mathbb I^n\times\{1\})\cap g''(b)(\mathbb I^n\times\{2\})=\varnothing$.
\end{itemize}
Condition $(4)$ means that $g''\in O_{\Lambda(f)}$. Therefore, $g=\Lambda^{-1}(g'')\in C(\mathbb I^n\times\{1,2\}\times B,E)$ is a fiber-preserving approximation of $f$ such that $g(\mathbb I^n\times\{1\}\times B)\cap g(\mathbb I^n\times\{2\}\times B)=\varnothing$. This provides the proof of Claim \ref{4}.

Therefore, $p$ has the \rm{FGP}-property.
\end{proof}

\begin{thm}
Let $p:E\to B$ be a locally compact $ANR$-fibration with compact fibers such that $B$ is a paracompact locally compact $C$-space. Then $p$ is a $Q$-manifold bundle if and only for every $b\in B$ we have:
\begin{itemize}
\item[(1)] $p^{-1}(b)$ has the disjoint disks property;
\item[(2)] $C(\mathbb I^n,p^{-1}(b))$ contains a dense set of homological $Z_n$-maps for each $n\geq 2$.
\end{itemize}
\end{thm}

\begin{proof}
Suppose $p$ is a $Q$-manifold bundle. Then there is a $Q$-manifold $M$ and an open cover $\mathcal U$ of $B$ such that
$p^{-1}(U)$ is homeomorphic to $U\times M$ for all $U\in\mathcal U$. Because for every $U\in\mathcal U$ and every $b\in U$ the fiber $p^{-1}(b)$ is a $Q$-manifold, conditions $(1)$ and $(2)$ are satisfied.

Suppose now that $p$ satisfies conditions $(1)$ and $(2)$. Then each $p^{-1}(b)$, $b\in B$, is a compact $ANR$ satisfying the disjoint disks property and for every $n\geq 3$ the space $C(\mathbb I^n,p^{-1}(b))$ contains a dense set of homological $Z_n$-maps. Therefore, by \cite[Theorem 2.9]{kv}, each fiber of $p$ is a $Q$-manifold.
We choose a local trivialization $\{U_\alpha:\alpha\in A\}$ of $p$ such that if $K\subset B$ is a compact set contained in some $U_\alpha$, then the fibration $p|p^{-1}(K):p^{-1}(K)\to K$ has the \rm{FGP}-property, see Proposition \ref{fibred_position}. Without loss of generality, we may assume that all $U_\alpha$ are $F_\sigma$-subsets of $B$, hence $C$-spaces. According to
\cite[Theorem 2.3]{tw}, it suffices to show that each $p|p^{-1}(U_\alpha)$ has the \rm{FGP}-property. Because every $U_\alpha$ is a locally compact $C$-space, we can assume that $p$ is a trivial fibration such that any fibration $p|p^{-1}(K):p^{-1}(K)\to K$, where $K\subset B$ is compact, has the \rm{FGP}-property. Therefore, we need to show that $p$ has the \rm{FGP}-property.
To this end, using the paracompactness of $B$, for every $i$ choose an open discrete family $\gamma_i=\{V_{s,i}:s\in S\}$ in $B$ such that $\bigcup_{i=1}^\infty\gamma_i$ is a cover of $B$
and each $\overline{V}_{s,i}$ is compact, see \cite{en}. Then the family
$\overline\gamma_i=\{\overline{V}_{s,i}:s\in S\}$ is also discrete in $B$ and $B_i=\bigcup\{\overline{V}_{s,i}:s\in S\}$ is closed in $B$.
 Since for every $i$ the fibrations $p|p^{-1}(\overline{V}_{s,i})$, $s\in S$,
have the \rm{FGP}-property, each $p|p^{-1}(B_i)$ also has the \rm{FGP}-property. Consequently, for every $n$ and $i$ the set
$\Theta_{n,i}$ of all $f\in C_P(\mathbb I^n\times\{1,2\}\times B_i,p^{-1}(B_i))$ with $f(\mathbb I^n\times\{1\}\times B_i)\cap f(\mathbb I^n\times\{2\}\times B_i)=\varnothing\}$ is dense in $C_P(\mathbb I^n\times\{1,2\}\times B_i,p^{-1}(B_i))$.
\begin{claim}
Each $\Theta_{n,i}$ is open in $C_P(\mathbb I^n\times\{1,2\}\times B_i,p^{-1}(B_i))$.
\end{claim}
Indeed, let $f_0\in\Theta_{n,i}$.
Using that  $\overline\gamma_i$ is discrete in $B$ and the pairs $F_{s,i}^1=f_0(\mathbb I^n\times\{1\}\times\overline{V}_{s,i})$, $F_{s,i}^2=f_0(\mathbb I^n\times\{2\}\times\overline{V}_{s,i})$ are compact disjoint subsets of $p^{-1}(B_i)$, we can find open subsets
$W_{s,i}^j$ and $G_{s,i}^j$ of $p^{-1}(B_i)$, $j=1,2$, with the following properties:
\begin{itemize}
\item[(6)] $F_{s,i}^j\subset W_{s,i}^j\subset\overline{W}_{s,i}^j\subset G_{s,i}^j$, $j=1,2$;
\item[(7)] For each $j=1,2$ the families $\beta_i^j=\{W_{s,i}^j:s\in S\}$ and $\eta_i^j=\{G_{s,i}^j:s\in S\}$ are discrete in $p^{-1}(B_i)$;
\item[(8)] $G_{s,i}^1\cap G_{s,i}^2=\varnothing$ for each $s$ and $i$.
\end{itemize}
Now, let $\mathcal G_i=\{G_{s,i}^1, G_{s,i}^2,p^{-1}(B_i)\setminus\overline W:s\in S\}$, where $\overline W=\bigcup\{\overline{W}_{s,i}^1\cup\overline{W}_{s,i}^2:s\in S\}$. Every $\mathcal G_i$ is an open cover of $p^{-1}(B_i)$ such that
$f\in B(f_0,\mathcal G_i)$ implies $f(\mathbb I^n\times\{1\}\times\{b\})\cap f(\mathbb I^n\times\{2\}\times\{b\})=\varnothing$ for every $b\in B_i$.
Hence, $B(f_0,\mathcal G_i)\subset\Theta_{n,i}$. This completes the proof of the claim because $B(f_0,\mathcal G_i)$ is a neighborhood of $f_0$ with respect to the limitation topology, see \cite{bo1}.

Since each $q_i:C_P(\mathbb I^n\times\{1,2\}\times B,E)\to C_P(\mathbb I^n\times\{1,2\}\times B_i,p^{-1}(B_i))$ is open (see Lemma \ref{open}),
$q_i^{-1}(\Theta_{n,i})$ is open and dense in $C_P(\mathbb I^n\times\{1,2\}\times B,E)$. According to Proposition \ref{baire}, $C_P(\mathbb I^n\times\{1,2\}\times B,E)$ has the Baire property. Therefore, the set
$\Theta_n=\bigcap_{i=1}^\infty q_i^{-1}(\Theta_{n,i})$ is dense in $C_P(\mathbb I^n\times\{1,2\}\times B,E)$. Finally, observe that
$g\in\Theta_n$ implies $g(\mathbb I^n\times\{1\}\times\{b\})\cap g(\mathbb I^n\times\{2\}\times\{b\})=\varnothing$ for every $b\in B$.
Therefore, $p$ has the \rm{FGP}-property.
\end{proof}

\section{$ANR$-fibrations over $C$-spaces}
The FGP property is necessary and sufficient for a locally compact $ANR$ fibration with compact fibers to be a $Q$-manifold bundle. In the present section we provide another condition which still guarantees that any locally compact $ANR$ fibration with compact fibers satisfying that condition is a $Q$-manifold bundle provided the base is a $C$-space.

A closed set $A\subset E$ is said to be a {\em fibred $Z_n$-set} if the set of all $f\in C_P(\mathbb I^n\times B,E)$ with $f(\mathbb I^n\times B)\cap A=\varnothing$ is dense in $C_P(\mathbb I^n\times B,E)$. We also define weak versions of fibred $Z_n$-sets: a closed set $A\subset E$ is a {\em weak fibred $Z_n$-set in $E$} (resp., {weak fibred homological $Z_n$-set in $E$})  if $A\cap p^{-1}(b)$ is a $Z_n$-set (resp., homological $Z_n$-set) in $p^{-1}(b)$ for all $b\in B$.
Fibred $Z_\infty$-sets (resp., weak fibred $Z_\infty$-sets or weak fibred homological $Z_\infty$-sets) will be called {\em fibred $Z$-sets} (resp., {\em weak fibred $Z$-sets or weak fibred homological $Z$-sets}). Finally, a map $f\in C_P(K\times B,E)$, where $K$ is a given space, is {\em fibred $Z$-map,
weak fibred $Z$-map or weak fibred homological $Z$-map} if the image $f(K)$ has the corresponding property.

\begin{thm}\label{fibred z}
Let $p:E\to B$ be a locally compact $ANR$-fibration with compact fibers such that $B$ is a $C$-space. Then every $b\in B$ has a basis of neighborhoods $U$ such that each
weak fibred $Z$-set in $p^{-1}(U)$ is a fibred $Z$-set in $p^{-1}(U)$.
\end{thm}
\begin{proof}
As in Proposition \ref{fibred_position}, where we were considering open $F_\sigma$ elements of $\mathcal U$, the proof is reduced to the case when $p$ is a trivial $ANR$-fibration such that
both fibrations $p':E\times Q\to E\to B$ and $p'':E\times Q\times [0,1)\to E\to B$ are trivial $Q$-manifold bundles. Consequently, there is a fiber-preserving homeomorphism
$h:E\times Q\times [0,1)\to M\times B$, where $M$ is a $Q$-manifold homeomorphic to $p^{-1}(b_0)\times Q\times [0,1)$ for some $b_0\in B$ (actually,
$M$ is homeomorphic to $p^{-1}(b)\times Q\times [0,1)$
for all $b\in B$). We consider $M$ as an open subset of a copy $Q_M$ of $Q$, and let $d$ be the usual convex metric on $Q_M$.
Suppose $A\subset E$ is a weak fibred $Z$-set in $E$. We need to show that for every $n$, a map $f\in C_P(\mathbb I^n\times B,E)$ and an open cover $\mathcal U$ of $E$ there is a map $f'\in C_P(\mathbb I^n\times B,E)$ such that
 $f'(\mathbb I^n\times B)\cap A=\varnothing$ and $f'$ is $\mathcal U$-close to $f$. To this end, fix such $f\in C_P(\mathbb I^n\times B,E)$ and an open cover $\mathcal U$ of $E$.
We identify 
$E$ with $E\times\{\overline 0\}\times\{0\}$, where $\overline 0=(0,0,...)\in Q$. Consider the retraction $r_E:E\times Q\times [0,1)\to E\times\{\overline 0\}\times 0$, $r_E((z_1,z_2,z_3))=(z_1,\overline 0,0)$. Under the above identifications, $r_E$ is a fiber-preserving retraction and $\mathcal V=h(r_E^{-1}(\mathcal U))$ is an open cover of $M\times B$. Moreover, consider the map
$\widetilde f:\mathbb I^n\times B\to M\times B$, $\widetilde f(x,b)=h(f(x,b),\overline 0,0))$. Since $h$ is a fiber-preserving homeomorphism,
$\widetilde f(x,b)$ is of the form $(f_M(x,b),b)$ for each $(x,b)\in\mathbb I^n\times B$.
Then, by Lemma \ref{limitation}, there is a bounded
continuous function $\varepsilon$ on $M\times B$ such that every map from the set
$$\{g\in C_P(\mathbb I^n\times B,M\times B):d_E(\widetilde f(y),g(y))<\varepsilon (\widetilde f(y)) \forall y\in\mathbb I^n\times B\}$$ is $\mathcal V$-close to $\widetilde f$, where
$d_E$ is the continuous pseudo-metric on $M\times B$ generated by $d$. 

Since $M$ is open in $Q_M$, the function
$\eta(x,b)=d\big(\pi_M(\widetilde f(x,b)), Q_M\backslash M\big)$ is positive and continuous on $\mathbb I^n\times B$ ($\pi_M:M\times B\to M$ is the projection.
Let $\delta:\mathbb I^n\times B\to (0,\infty)$ be the function defined by $\delta(x,b)=\min\{\eta(x,b),\varepsilon(\widetilde f(x,b))\}$.
\begin{claim}\label{1}
There is a bounded continuous function $\alpha:B\to (0,\infty)$ such that $\alpha(b)<\delta(x,b)$ for each $(x,b)\in\mathbb I^n\times B$.
\end{claim}
Indeed, for every $b\in B$ let $m_b=\min\{\delta(x,b):x\in\mathbb I^n\}$. Then each $m_b>0$ and consider the set-valued map
$\theta:B\rightsquigarrow \mathcal F_c((0,\infty))$, $\theta(b)=(0,m_b]$. This map is lower semi-continuous and, by \cite[Theorem 6.2, p.116]{rs},
$\theta$ has a continuous selection $\alpha$. Clearly, $\alpha$ has the required property.


Define the set-valued map
$\varphi:\mathbb I^n\times B\to\mathcal F_c(Q_M)$  by
$$\varphi(x,b)=\{u\in Q_M:d(f_M(x,b),u)\leq\alpha(b)\}.$$
Observe that each $\varphi(x,b)$ is a compact convex subset of $M$ because $\alpha(b)<\eta(x,b)$ for all $(x,b)\in \mathbb I^n\times B$. Since
$f_M$ and $\alpha$ are continuous, $\varphi$ is lower semi-continuous.
Recall that each $A(b)=A\cap p^{-1}(b)$ is a $Z$-set in $p^{-1}(b)$. So is
$A(b)\times Q\times [0,1)$ in the set $p^{-1}(b)\times Q\times [0,1)$. This implies that $H(b)=h(A(b)\times Q\times [0,1))$ is a $Z$-set in
$h(p^{-1}(b)\times Q\times [0,1))$. On the other hand, $h(p^{-1}(b)\times Q\times [0,1))=M\times\{b\}$.
Therefore,
$\psi(b)=\pi_M\big(H(b))$ is a $Z$-set in $M$ for every $b\in B$.

\begin{claim}\label{2}
The set-valued map $\psi:B\rightsquigarrow\mathcal F(M)$  has a closed graph and
$\psi(b)\cap\varphi(x,b)$ is a $Z$-set in $\varphi(x,b)$ for all $(x,b)\in\mathbb I^n\times B$.
\end{claim}
To show that the graph
$Gr(\psi)$ is closed in $B\times M$, take a net $\{(b_\gamma,u_\gamma)\}$ in $Gr(\psi)$ converging to a point $(b^*,u^*)\in B\times M$.
Then for each $\gamma$ we have $u_\gamma\in\psi(b_\gamma)$, so  $(u_\gamma,b_\gamma)\in h(A(b)\times Q\times [0,1))$. This means that
there exists $z_\gamma\in A(b_\gamma)\times Q\times [0,1)$ with $h(z_\gamma)=(u_\gamma,b_\gamma)$. Since the net $\{(u_\gamma,b_\gamma)\}$ converges to $(u^*,b^*)$ in $M\times B$, $\{z_\gamma\}$ converges to $z^*=h^{-1}(u^*,b^*)$ in $E\times Q\times [0,1)$. Since $A\times Q\times [0,1)$ is closed in $E\times Q\times [0,1)$ and contains all $z_\gamma$, $z^*$ belongs to $A\times Q\times [0,1)$. On the other hand,
 $z^*\in p^{-1}(b^*)\times Q\times [0,1)$ because $h$ is fiber-preserving. Hence, $z^*\in (A\times Q\times [0,1))\cap(p^{-1}(b^*)\times Q\times [0,1))$. The last intersection is $A(b^*)\times Q\times [0,1)$, so $h(z^*)=(u^*,b^*)\in H(b^*)$. Therefore, $u^*\in\psi(b^*)$, or equivalently,
 $(b^*,u^*)\in Gr(\psi)$. This shows that $Gr(\psi)$ is closed in $B\times M$.

To prove the second part of Claim \ref{2}, for each
$(x,b)\in\mathbb I^n\times B$ consider the open in $M$ set
$O(x,b)=\{u\in Q_M:d(f_M(x,b),u)<\alpha(b)\}$. Because $\psi(b)$ is a $Z$-set in $M$, $\psi(b)\cap O(x,b)$ is a $Z$-set in $O(x,b)$. Finally, since $\varphi(x,b)$ is the closure of $O(x,b)$, $\psi(b)\cap\varphi(x,b)$ is a $Z$-set in $\varphi(x,b)$.

Since $\psi$ has a closed graph, so has the map $\psi':I^n\times B\rightsquigarrow\mathcal F(M)$, $\psi'(x,b)=\psi(b)$.
Because $\mathbb I^n\times B$ is a paracompact $C$-space, we can apply Lemma \ref{gv} to obtain a map $g_M:\mathbb I^n\times B\to M$ with
$g_M(x,b)\in\varphi(x,b)\backslash\psi'(x,b)$. Then $(g_M(x,b),b)$ defines a map $\widetilde g\in C_P(\mathbb I^n\times B,M\times B)$ with
$\widetilde g(x,b)\in M\times\{b\}\backslash H(b)$. So,
$g(x,b)\in (p^{-1}(b^*)\backslash A(b))\times Q\times [0,1)$ for all $(x,b)\in\mathbb I^n\times B$, where $g\in C_P(\mathbb I^n\times B\to E\times Q\times [0,1))$ is the map defined by $g=h^{-1}\circ\widetilde g$. Because $d(f_M(x,b),g_M(x,b))\leq\alpha(b)<\varepsilon(\widetilde f(x,b))$,
$\widetilde f$ and $\widetilde g$ are $\mathcal V$-close maps. Consequently, $(f(x,b),\overline 0,0)$ and $g(x,b)$ are $r_E^{-1}(\mathcal U)$-close for each $(x,b)$. This implies that $f$ is $\mathcal U$-close to the map $f'\in C_P(\mathbb I^n\times B\to E)$, $f'=r_E\circ g$.
Finally, observe that $g(x,b)\in (p^{-1}(b^*)\backslash A(b))\times Q\times [0,1)$ yields $f'(x,b)\in p^{-1}(b^*)\backslash A(b)$ for all $(x,b)$. Hence, $f'(\mathbb I^n\times B)\cap A=\varnothing$.
 This completes the proof of Theorem \ref{fibred z}.
\end{proof}

\begin{thm}\label{*}
Let $p:E\to B$ be a locally compact $ANR$-fibration with compact fibers. Then $p$ is a $Q$-manifold bundle provided the following holds:
\begin{itemize}
\item[(*)] Every point of $B$ has a basis consisting of $C$-spaces $U$ such that each $C_P(\mathbb I^n\times U,p^{-1}(U))$, $n\geq 1$, contains a dense set of weak fibred $Z$-maps.
\end{itemize}
\end{thm}

\begin{proof}
We take a local trivialization $\mathcal U$ of $p$ such that each $U\in\mathcal U$ is a $C$-space and for every $n$ the space
$C_P(\mathbb I^n\times U,p^{-1}(U))$ contains a dense set of weak fibred $Z$-maps. Moreover, by Theorem \ref{fibred z}, we can also assume that
each $U\in\mathcal U$ satisfies the additional condition:
\begin{itemize}
\item Every weak fibred $Z$-set in $p^{-1}(U)$ is a fibred $Z$-set in $p^{-1}(U)$.
\end{itemize}
Let show that each $p_U=p|p^{-1}(U)$, $U\in\mathcal U$, has the \rm{FGP}-property.
It suffices to prove that for every $n$ every weak fibred $Z$-map
$f:\mathbb I^n\times U\to p^{-1}(U)$ can be approximated by fiber-preserving maps $f':\mathbb I^n\times U\to p^{-1}(U)$ such that
 $f'(\mathbb I^n\times U)\cap f(\mathbb I^n\times U)=\varnothing$. This follows from the fact that
 the set $A=f(\mathbb I^n\times U)$ is a fibred $Z$-set in $p^{-1}(U)$ because $f$ is weak fibred $Z$-map.
\end{proof}


\end{document}